\documentclass[11pt, a4paper]{amsart}
\usepackage{amssymb,latexsym,amsmath,amsfonts,amsthm}

\usepackage{tikz}
\usetikzlibrary{arrows}
\usetikzlibrary{decorations.markings}
\tikzset{->-/.style={decoration={
  markings,
  mark=at position .5 with {\arrow{>}}},postaction={decorate}}}

\newcommand{\cev}[1]{\reflectbox{\ensuremath{\vec{\reflectbox{\ensuremath{#1}}}}}}

\numberwithin{equation}{section}

\usepackage[mathscr]{eucal}

\usepackage[T1]{fontenc}
\usepackage{txfonts}
\usepackage[pdftex,bookmarks=true]{hyperref}
\usepackage{graphicx}
\usepackage{enumerate}

\usepackage[all]{xy}

\usepackage{enumerate}

\theoremstyle{plain}
\newtheorem{theorem}{Theorem}[section]

\newtheorem{lemma}[theorem]{Lemma}
\newtheorem{prop}[theorem]{Proposition}

\theoremstyle{remark}
\newtheorem{rmk}[theorem]{Remark}
\newtheorem{rmks}[theorem]{Remarks}

\theoremstyle{definition}
\newtheorem{dfn}[theorem]{Definition}
\newtheorem{question}{Question}

\author{Bidyut Sanki}
\address{
Institute of Mathematical Sciences\\ 
CIT Campus, Tharamani  \\ 
Chennai, 600113\\
India}
\email{bidyut.iitk7@gmail.com}

\thanks{The author has been supported by the Post Doctoral Fellowship funded by a J. C. Bose fellowship of Prof. Mahan Mj.}

\begin{document}
\title{Embedding of metric graphs on hyperbolic surfaces}




\begin{abstract}
An embedding of a metric graph $(G, d)$ on a closed hyperbolic surface is \emph{essential}, if each complementary region has a negative Euler characteristic. We show, by construction, that given any metric graph, its metric can be rescaled so that it admits an essential and isometric embedding on a closed hyperbolic surface. The essential genus $g_e(G)$ of $(G, d)$ is the lowest genus of a surface on which such an embedding is possible. In the next result, we establish a formula to compute $g_e(G)$. Furthermore, we show that for every integer $g\geq g_e(G)$, $(G, d)$ admits such an embedding (possibly after a rescaling of $d$) on a surface of genus $g$.

Next, we study minimal embeddings where each complementary region has Euler characteristic $-1$. The maximum essential genus $g_e^{\max}(G)$ of $(G, d)$ is the largest genus of a surface on which the graph is minimally embedded. Finally, we describe a method explicitly for an essential embedding of $(G, d)$, where $g_e(G)$ and $g_e^{\max}(G)$ are realized.
\end{abstract}

\maketitle


\section{Introduction}

Graphs on surfaces play an important role in the study topology and geometry of surfaces. A $2$-cell embedding of a graph $G$ on a closed oriented \emph{topological surface} $S_g$ of genus $g$ is a cellular decomposition of $S_g$, whose $1$-skeleton is isomorphic to $G$~\cite{Duke}. In topological graph theory, characterization of the surfaces on which a graph can be $2$-cell embedded is a famous problem and well studied~\cite{SS}. In this direction, Kuratowski is the first who has shown that a graph is planar if and only if it does not contain $K_{3,3}$ or $K_5$ as a minor, where $K_{3,3}$ is the complete bipartite graph with $(3,3)$ vertices and $K_5$ is the complete graph with $5$ vertices. Hence, these are the only minimal non-planar graphs. 

The genus of a surface $S$ is denoted by $g(S)$. The genus of a graph $G$ is defined by $g(G)=\min\{g(S)\}$, where the minimum is taken over the surfaces $S$ on which $G$ is $2$-cell embedded. The maximum genus $g_M(G)$ is similarly defined~\cite{NHX}. 

In~\cite{Duke}, Duke has shown that every finite graph $G$ admits a $2$-cell embedding on a surface $S_g$ of genus $g$ for each $g(G)\leq g \leq g_{M}(G)$. A $2$-cell embedding of a graph realizing its genus is called a minimal embedding. A maximal embedding is defined similarly. In Theorem 3.1~\cite{Duke}, Duke  has derived a sufficient condition for a $2$-cell embedding to be non-minimal and provided an algorithm to obtain an embedding on a lower genus surface. 

The maximum genus problem has studied by Xuong in~\cite{NHX}. In Theorem 3~\cite{NHX}, Xuong has obtained the formula $g_M(G)=\frac{1}{2}(\beta(G) - \zeta(G))$ for maximal genus, where $\beta(G)$ and $\zeta(G)$ are the Betti number and Betti deficiency of $G$ respectively. Furthermore, in a maximal embedding, the number of $2$-cells in the cellular decomposition is $1+\zeta(G).$ For more results on $2$-cell embedding, we refer to~\cite{SS} and~\cite{RR}.

A Riemann surface is a surface equipped with a complex structure. In this paper, when we say surface, we will always mean a Riemannian surface with constant negative sectional curvature -1. Such a surface is called a hyperbolic surface.

Configuration of geodesics on hyperbolic surfaces has become increasingly important in the study of mapping class groups and the moduli spaces of surfaces through the \emph{systolic} function~\cite{PSS} and \emph{filling} pair length function~\cite{Aougab} in particular. The study of filling systems has its origin in the work of Thurston~\cite{Thurston}. The set $\chi_g$, consisting of the closed hyperbolic surfaces of genus $g$ whose systoles fill, is a so-called \emph{Thurston set}. A closed hyperbolic surface with a pair of pants decomposition of bounded lengths is called a \emph{trivalent surface} (see Section 4~\cite{Anderson}). Recently, Anderson, Parlier and Pettet~\cite{Anderson} have studied the shape of $\chi_g$ comparing with the set $Y_g$ of trivalent surfaces by giving a lower bound on the Hausdorff distance between them in the moduli space $\mathcal{M}_g$.      

There is a natural connection between graphs and surfaces. For instance, given a system of curves on a surface, the union forms a so-called \emph{fat graph}, where the intersection points are the vertices, sub-arcs between the intersections are the edges and the cyclic order on the set of edges incident at each vertex is determined by the orientation of the surface. In~\cite{Balacheff}, Balacheff, Parlier and Sabourau have studied the geometry and topology of Riemann surfaces by embedding a suitable graph on the surface which captures some of its geometric and topological properties. 

In this paper, we will be interested in studying graphs on hyperbolic surfaces, whose edges are realized by geodesic segments. Furthermore, no two edges meet at their interior. Such a graph has a metric, where the distances between points on the graph are measured along a shortest path in the induced metric on the graph. A graph $G$ on a closed surface $S$ is \emph{essential}, if each component of $S\setminus G$ has a negative Euler characteristic. In this paper, by a graph, we always mean a \emph{finite} and \emph{connected} graph. In particular, we will be looking at (finite and connected) metric graphs.

\begin{dfn}
A metric graph $(G, d)$ is a pair of a graph $G$ and a positive real valued function $d: E \to \mathbb{R}_+$ on the set $E$ of edges of $G$.
\end{dfn}

The central questions in this paper are the following:

\begin{question}\label{central_question}
Given a metric graph $(G,d)$.

\begin{enumerate}
\item Does there exist a closed hyperbolic surface on which $(G, d)$ can be essentially embedded$?$

\item Characterize the surfaces on which such an embedding of $(G, d)$ is possible.\\What is the lowest genus of such a surface$?$
\end{enumerate}
\end{question}

While studying embeddings of metric graphs, we are of course interested in isometric embeddings, i.e., an injective map $\Phi: (G,d)\to S$ which preserves the lengths of the edges. An isometric embedding $\Phi: (G, d)\to S$ is called \emph{essential}, if $\Phi(G)$ is essential on $S$.

\subsection*{Scaling of metric} 
Given a metric graph $(G, d)$ and a positive real number $t$, define 
\begin{eqnarray*}
d_t: E\to \mathbb{R}_+ \text{ by } d_t(e)= td(e) \text{ for all } e\in E.
\end{eqnarray*}
Then $d_t$ is the metric obtained from $d$ scaling by $t$. Perhaps, the more natural question is to ask all these up to scaling. An obstruction in the un-scaled case is the \emph{Margulis lemma} (Corollary 13.7 in~\cite{FM}). Therefore, the general question is as follows: Given a metric graph $(G, d)$, does there exist a $t\in \mathbb{R}_+$ such that $(G, d_t)$ can be embedded essentially and isometrically on a closed hyperbolic surface? From now onwards, by an embedding of a metric graph $(G, d)$, we mean an essential and isometric embedding of $(G, d_t)$ for some $t>0$.

The first result, we obtain, is stated below which answers (1) in Question~\ref{central_question}.

\begin{theorem}\label{thm:1}
Given a metric graph $(G, d)$ with degree of each vertex at least three, there exists a closed hyperbolic surface $S_g$ of genus $g=|E|+\beta(G)$ on which $(G,d)$ is embedded, where $\beta(G)$ and $|E|$ are the Betti number and the number of edges of $G$ respectively. 
\end{theorem}

\subsection*{Notation.} 
$S(G,d)$ denotes the set of surfaces on which the metric graph $(G, d)$ admits an essential and isometric embedding possibly after rescaling its metric. 

Now, we focus on the genera of the surfaces in $S(G,d)$. 

\begin{dfn}
The essential genus $g_e(G)$ of a metric graph $(G, d)$ is defined by $$g_e(G)=\min \{ g(S)|\; S\in S(G, d)\}.$$  
\end{dfn}

If $T$ is a spanning (or maximal) tree of a graph $G$, then $\xi(G, T)$ denotes the number of components in $G\setminus E(T)$ with an odd number of edges, where $E(T)$ denotes the set of edges of $T$.

\begin{dfn} 
The Betti deficiency of a graph $G$, denoted by $\zeta(G)$, is defined by
\begin{equation}\label{eq:7.5}
\zeta(G)= \min\{\xi(G, T)| \textit{ $T$ is a spanning tree of $G$}\}.
\end{equation}
\end{dfn}

We prove the theorem, stated below, which computes the essential genus of a metric graph and thus answers (2) in Question~\ref{central_question}.

\begin{theorem}\label{thm:2}
The essential genus of a metric graph $(G, d)$ is given by $$g_{e}(G)=\frac{1}{2}(\beta(G)-\zeta(G))+2q+r,$$ where $\beta(G), \zeta(G)$ are the Betti number and Betti deficiency of $G$ respectively and $q, r$ are the unique integers satisfying $\zeta(G)+1=3q+r$, $0\leq r < 3$. Furthermore, for any given $g\geq g_e(G)$, there exists a closed hyperbolic surface $F$ of genus $g$ on which $(G,d)$ can be essentially embedded.
\end{theorem}

An embedding of $(G,d)$ on a hyperbolic surface $S$ is the simplest, if the Euler characteristic of each complementary region is $-1$ and hence, we define \emph{minimal embedding} as follows.

\begin{dfn}
An embedding $\Phi: (G,d) \to S$ is called minimal, if $\chi(\Sigma)=-1$ for each component $\Sigma$ in $S\setminus \Phi(G)$.
\end{dfn}

Given a metric graph, there exists a minimal embedding, where the essential genus is realized. Note that, the essential genus can also be realized by a non-minimal embedding. For instance, the complement might contains a torus with two boundary components. The set of closed hyperbolic surfaces, on which $(G,d)$ can be minimally embedded, is denoted by $S_m(G,d)$. The genera of the surfaces in $S_m(G,d)$ are bounded from below by $g_e(G)$ and this bound is sharp. We define $$g_e^{\max}(G)=\max\{g(S)| \; S\in S_m(G,d)\}.$$ It is a fact that $g^{\max}_e(G)\leq 1/2\left( \beta(G) +1 + 2|E|/T(G)\right)$, follows from Euler's equation (for instance, we refer to~\cite{Beineke}). Here, $T(G)$ is the \emph{girth} of the graph $G$.

Now, we focus on an explicit construction of minimal (or maximal) embedding as this is preferred over random constructions. To embed a graph on a surface minimally (or maximally), the crucial part is to find a suitable \emph{fat graph structure} which gives the minimum (or maximum) number of boundary components among all possible fat graph structures on the graph. For a definition of fat graphs, we refer to Definition~\ref{def:fat}. For a fat graph structure $\sigma_0$ on $G$, the number of boundary components in $(G, \sigma_0)$ is denoted by $\#\partial(G, \sigma_0)$.

We prove the following proposition which leads to an algorithm for minimal and maximal embeddings. Furthermore, given any integer $g$ satisfying $g_e(G)\leq g \leq g_e^{\max}(G)$, there exists a closed hyperbolic surface of genus $g$ on which $(G, d)$ can be minimally embedded and thus answers Question~\ref{central_question}.

\begin{prop}\label{thm:4}
Let $G=(E, \sim, \sigma_1)$ be any graph with degree of each vertex at least three. Suppose, $\sigma_0=\prod\limits_{v\in V}\sigma_v$ be a fat graph structure on $G$ such that there is a vertex $v$ which is common in $b\;(\geq 3)$ boundary components. Then there exists a fat graph structure $\sigma'_{0}$ on $G$, such that  $$\#\partial(G, \sigma_0') = \# \partial(G, \sigma_0) - 2.$$
\end{prop}



\section{Preliminaries}\label{sec:2}
In this section, we recall some graph theory and geometric notions. Also, we develop a lemma which is essential in the subsequent sections. 

\subsection{Fat graph}\label{Fat graph} 
Before going to the formal definition of fat graph (ribbon graph), we  recall a definition of graph and a few graph parameters. The definition of graph we use here, is not the standard one which is used in ordinary graph theory. But, it is straightforward to see that this definition is equivalent to the standard one.

\begin{dfn}
A finite graph is a triple $G=(E_1, \sim, \sigma_1)$, where $E_1$ is a finite, non-empty set with an even number of elements, $\sigma_1$ is a fixed-point free involution on $E_1$ and $\sim$ is an equivalence relation on $E_1$.
\end{dfn}

In ordinary language, $E_1$ is the set of directed edges, $E:=E_1/\sigma_1$ is the set of undirected edges and $V:=E_1/\!\!\sim$ is the set of vertices. The involution $\sigma_1$ maps a directed edge to its reverse directed edge. If $\vec{e}\in E_1$, we say that $\vec{e}$ is emanating from the vertex $v=[\vec{e}]$, the equivalence class of $\vec{e}$. The degree of a vertex $v\in V$ is defined by $\deg(v)=|v|$. 

The \emph{girth} $T(G)$ of a graph $G$ is the length of a shortest non-trivial simple cycle, where the length of a cycle is the number of edges it contains. Furthermore, the girth of a tree (graph without a simple cycle) is defined to be infinity. The \emph{Betti number} of $G$, denoted by $\beta(G)$, is defined by $\beta(G)= -|V|+|E|+1$. 

Now, we define fat graphs. Informally, a \emph{fat graph} is a graph equipped with a cyclic order on the set of directed edges emanating from each vertex. If the degree of a vertex is less than three, then the cyclic order is trivial. Therefore, we consider the graphs with degree of each vertex at least three.  

\begin{dfn}\label{def:fat}
A fat graph is a quadruple $G=(E_1, \sim, \sigma_1, \sigma_0)$, where
\begin{enumerate}
\item $(E_1, \sim, \sigma_1)$ is a graph and 
\item $\sigma_0$ is a permutation on $E_1$ so that each cycle corresponds to a cyclic order on the set of oriented edges emanating from a vertex.
\end{enumerate}
\end{dfn}

For a vertex $v$ of degree $d$, $ \sigma_{v} = (e_{v,1}, e_{v,2}, \dots, e_{v,d})$ represents a cyclic order on $v$, where $e_{v, i}, i=1,2, \dots, d$ are the directed edges emanating from the vertex $v$ and $\sigma_0=\prod_{v\in V}\sigma_v$. Given a fat graph $G$, we can construct an oriented topological surface $\Sigma(G)$ with boundary by thickening its edges. The number of boundary components in $\Sigma(G)$ is the number of disjoint cycles in $\sigma_1*\sigma_0^{-1}$ (see~\cite{BS}, Section 2.1). For more details on fat graphs, we refer to~\cite{BS},~\cite{Bidyut} and~\cite{AR}.

\subsection{Pair of pants}\label{pants} A hyperbolic 3-holed sphere is called a pair of pants. It is a fact in hyperbolic geometry that given any three positive real numbers $l_1, l_2$ and $l_3$, there exists a unique pair of pants with boundary geodesics of lengths $l_1, l_2$ and $l_3$ (see Section 3.1 in~\cite{PB}). Let $P_x$ be the pair of pants with boundary geodesics of lengths $1, 1$ and $2x$, where $x\in \mathbb{R}_+$. We define a function $f$, where for $x\in \mathbb{R}_+$, $f(x)$ is the distance between the two boundary components of $P_x$ of length $1$. 

\begin{lemma}\label{lem:1}
The function $f$ is continuous and strictly monotonically increasing with $$f(\mathbb{R}_+)=(f_{min}, \infty), \text{ where } f_{\min} = \cosh^{-1}\left( \frac{\cosh^2\frac{1}{2}+1}{\sinh^2 \frac{1}{2}} \right).$$ 
\end{lemma}

\begin{proof}
The distance $f(x)$ between the boundary components of length $1$ is realized by the common perpendicular geodesic segment to these boundary geodesics. 

The common perpendicular geodesic segments between the pair of distinct boundary components of $P_x$ decompose it into two isometric right angled hexagons with alternate sides of lengths $\frac{1}{2}, \frac{1}{2}$ and $x$. Now, using formula (i) in Theorem 2.4.1~\cite{PB}, we have $$f(x)=\cosh^{-1}\left( \frac{\cosh^2\frac{1}{2}+\cosh x}{\sinh^2 \frac{1}{2}} \right),$$ which implies the lemma. 
\end{proof}


\section{Essential embedding of metric graph}\label{sec:3}
In this section, we prove Theorem~\ref{thm:1}. Note that, if a graph $G$ is a cycle, then it is easy to see that the metric graph $(G, d)$ can be essentially and isometrically embedded on any hyperbolic surface for any metric $d$, possibly after rescaling the metric. Therefore, in the remaining part of this section, we exclude the case, where the graph is a cycle. 

\begin{dfn}
A metric graph is called geometric, if it can be essentially and isometrically embedded on a closed hyperbolic surface.
\end{dfn}

Let $(G, d)$ be a metric graph with degree of each vertex at least two, where $G=(E_1, \sim, \sigma_1)$. If $v=\{\vec{e}_i: i=1, 2\}$ is a vertex with degree $2$, then we define a new graph $G'=(E_1', \sim', \sigma_1')$ with metric $d'$ by removing the vertex $v$ and replacing two edges $e_i=\{\vec{e}_i, \cev{e}_i\}, i=1, 2,$ by a single edge $e=\{\vec{e}, \cev{e}\}$ in $G$ (see Figure~\ref{fig:3}). The metric $d'$ is defined by, $d'(x)=d(x), \;\textit{for all}\;\; x \in E'\setminus \{e\}\;\;\textit{and}\;\;
d'(e) = d(e_1) + d(e_2).$

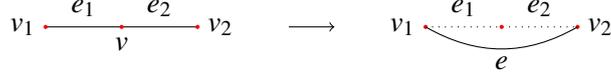
\begin{figure}[htbp]
\begin{center}
\begin{tikzpicture}
\draw [dotted] (-1, 0)node [left] {$v_1$} -- (0,0) -- (1, 0)node [right] {$v_2$}; \draw [fill, red] (1, 0) circle [radius=0.02]; \draw [fill, red] (0, 0) circle [radius=0.02]; \draw [fill, red] (-1, 0) circle [radius=0.02]; 
\draw (-1, 0) to [bend right]  (1, 0);
\draw (0.5, 0)node [above] {$e_2$}; \draw (-0.5, 0)node [above] {$e_1$}; \draw (0, -0.25)node [below] {$e$};
\draw [->] (-2.8, 0) -- (-2.2, 0);
\draw (-6, 0)node [left] {$v_1$} -- (-5,0)node [below] {$v$} -- (-4, 0)node [right] {$v_2$}; \draw [fill, red] (-6, 0) circle [radius=0.02]; \draw [fill, red] (-5, 0) circle [radius=0.02]; \draw [fill, red] (-4, 0) circle [radius=0.02]; \draw (-5.5, 0)node [above] {$e_1$}; \draw (-4.5, 0)node [above] {$e_2$}; 
\end{tikzpicture}
\end{center}
\caption{Replacement of two edges by a single edge and removal of a vertex.}\label{fig:3}
\end{figure}

\begin{lemma}\label{lem1}
A graph $(G,d)$ is geometric if and only if $(G', d')$ is geometric. Moreover, the essential genera of these graphs are same, i.e., $g_e(G)=g_e(G').$
\end{lemma}

In light of Lemma~\ref{lem1}, from now, we assume that the degree of each vertex of the graph $G$ is at least three, i.e., $\deg(v)\geq 3,\;\textit{for all}\;\; v\in V.$

 \begin{proof}[Proof of Theorem~\ref{thm:1}]
Let $(G, d)$ be a given metric graph with degree of each vertex at least three. For each vertex $v$, we assign a hyperbolic $\deg(v)$-holed sphere $S(v)$ with each boundary geodesic of length $1$. Namely, we construct $2\deg(v)$-sided right-angled hyperbolic polygon $P(v)$ with one set of alternate sides of length $\frac{1}{2}$ by attaching $2\deg(v)$ copies of $Q(\pi/\deg(v))$. Here, $Q(\theta)$ denotes a sharp corner (also known as Lambert quadrilateral), whose only angle not equal to right angle is $\theta$ and a side opposite to this angle is of length $\frac{1}{4}$ (see Figure~\ref{fig:1.1}). Then consider two copies of $P(v)$ and glue them in an obvious way by isometries to obtain $S(v)$. 

Consider a central point of $P(v)$'s on $S(v)$ and connect it by the distance realizing geodesic segment to each boundary component which meets orthogonally. Now, applying formula (vi) in Theorem 2.3.1~\cite{PB}, on the sharp corner $Q(\pi/\deg(v))$ as indicated in Figure~\ref{fig:1.1}, we have the length of the perpendicular geodesic segment $$x_v=\sinh^{-1}\left( \coth(1/4)\coth (\pi/\deg(v))\right).$$ 

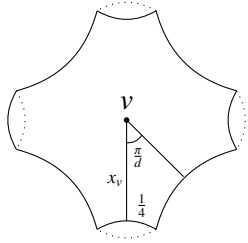
\begin{figure}[htbp]
\begin{center}
\begin{tikzpicture}
\draw (1.5*0.966, 1.5*0.259) to [bend left] (1.5*0.259, 1.5*0.966); \draw (-0.966*1.5, 0.259*1.5) to [bend right] (-0.259*1.5, 0.966*1.5); \draw (-0.966*1.5, -0.259*1.5) to [bend left] (-0.259*1.5, -0.966*1.5); \draw (0.966*1.5, -0.259*1.5) to [bend right] (0.259*1.5, -0.966*1.5); 

\draw (0.259*1.5, 0.966*1.5) to [bend left] (-0.259*1.5, 0.966*1.5);\draw [dotted] (0.259*1.5, 0.966*1.5) to [bend right] (-0.259*1.5, 0.966*1.5); 

\draw [dotted](0.259*1.5, -0.966*1.5) to [bend left] (-0.259*1.5, -0.966*1.5);\draw (0.259*1.5, -0.966*1.5) to [bend right] (-0.259*1.5, -0.966*1.5); 

\draw [dotted] (0.966*1.5, 0.259*1.5) to [bend left] (0.966*1.5,-0.259*1.5);\draw (0.966*1.5, 0.259*1.5) to [bend right] (0.966*1.5,-0.259*1.5);

\draw [dotted] (-0.966*1.5, 0.259*1.5) to [bend left] (-0.966*1.5,-0.259*1.5);\draw (-0.966*1.5, 0.259*1.5) to [bend right] (-0.966*1.5,-0.259*1.5);

\draw [fill] (0,0)node [above] {$v$} circle (0.03cm);

\draw (0,0) -- (0, -1.34); \draw (0, 0) -- (0.75, -0.75); 
\draw (0, -0.3) to [bend right] (0.2, -0.2); \draw (0.15, -0.5)node {\tiny $\frac{\pi}{d}$}; \draw (-0.15, -0.8)node {\tiny $x_v$}; \draw (0.2, -1.13)node {\tiny $\frac{1}{4}$};
\end{tikzpicture}
\end{center}
\caption{Hyperbolic $d$-holed sphere, $d=4.$}\label{fig:1.1}
\end{figure}

For an edge $e$ with the ends $u$ and $v$, we define $l(e)=x_u+x_v$. Now, we choose a positive real number $t$ such that $d_t(e)> l(e)+f_{\min}$ for all edges $e$ in $G$, where $f_{min}$ is given in Lemma~\ref{lem:1}. 

To each edge $e$, we assign a pair of pants $P_{x_e}$ (as in Section~\ref{pants}, Lemma~\ref{lem1}), where $x_e\in \mathbb{R}$ satisfying $f(x_e) = d_t(e)-l(e)$.

Now, glue the surfaces $S(v), v\in V$ and $P_{x_e}, e\in E$ along the boundaries of length $1$ according to the graph $G$ with twists so that all of the distance realizing orthogonal geodesic segments meet. Thus, we obtain a surface, denoted by $\Sigma_{\partial}(G,d)$, with boundary on which $(G, d_t)$ is isometrically embedded. We turn our surface into a closed surface $\Sigma(G,d)$ by attaching one-holed tori to the boundary components.     

Finally, we count the genus $g$ of $\Sigma(G,d)$ by counting the number of pairs of pants in a pants decomposition which gives $2g-2=\sum_{v\in V}(\deg(v)-2) +2|E|$. This equation and the relation $\sum\limits_{v\in V} \deg(v) = 2|E|$ conclude the proof. 
\end{proof}


\section{Fat graph structures and embeddings}\label{sec:3.5}
In this section, we consider a fat graph structure $\sigma_0$ on a given graph $(G, d)$ and construct a closed hyperbolic surface $S(G, d_t, \sigma_0)$ on which the metric graph $(G, d_t)$, for some $t> 0$, essentially and isometrically embedded. The genus of $S(G, d_t, \sigma_0)$ depends on $\sigma_0$. 

A graph $G$ on an oriented surface $S$ has a natural fat graph structure $\sigma_0$ determined by the orientation of $S$. Conversely, if $\sigma_0$ is any given fat graph structure on $G$, then there exists an essential and isometric embedding of $(G, d_t)$, for some $t>0$, on a closed hyperbolic surface $S$ of genus $g=|E|+\beta(G)$ (see Theorem~\ref{thm:1}), where the fat graph structure $\sigma_0$ is realized. Construction of such an embedding follows the similar procedure as in the proof of Theorem~\ref{thm:1}. Here, the only difference is that one needs to glue the building blocks hyperbolic $\deg(v)$-holed spheres $S(v)$, $v\in V$ and $P_{x_e}, e\in E$ according to the fat graph structure $\sigma_0$. 

\subsection{Embedding on a surface with totally geodesic boundary} 
Let $N_\epsilon(G, d_t, \sigma_0)$ be the regular (tubular) $\epsilon$ neighborhood of $G$ on $S$, where $\epsilon>0$ is sufficiently small. Let $\beta'$ be a boundary component of $N_\epsilon(G, d_t, \sigma_0)$. Then $\beta'$ is an essential simple closed curve on $S$ as the graph is essentially embedded (in particular, no coplementary region is a disc). Therefore, there is a unique geodesic representative $\beta$ (simple and closed) in its free homotopy class.

We obtain a hyperbolic surface $\Sigma_0(G, d_t , \sigma_0)$ with totally geodesic boundary by cutting the surface $S$ along the simple closed geodesics in the free homotopy classes of the boundary components of $N_\epsilon(G, d_t, \sigma_0)$.

\begin{lemma}
The metric graph $(G, d_t)$ is isometrically embedded on the hyperbolic surface $\Sigma_0(G, d_t , \sigma_0)$ with totally geodesic boundary. Furthermore, the number of boundary coponents of $\Sigma_0(G, d_t , \sigma_0)$ is the number of orbits of $\sigma_1*\sigma_0^{-1}$.
\end{lemma}

\subsection{Embedding on a closed hyperbolic surface} 
In this subsection, we cap the surface $\Sigma_0(G, d_t, \sigma_0)$ by hyperbolic surfaces with boundary to obtain a closed surface. Equivalently, we embed $\Sigma_0(G, d_t, \sigma_0)$ isometrically and essentially on a closed hyperbolic surface. We describe two gluing procedures below.

\subsubsection{Glue I}\label{G1} 
In this gluing procedure, we assume that $\Sigma_0(G, d_t, \sigma_0)$ has at least three boundary components and choose any three of them, say $\beta_1, \beta_2$ and $\beta_3$. We consider a pair of pants $Y$, with boundary geodesics $b_1, b_2$ and $b_3$ of lengths $l(\beta_1), l(\beta_2)$ and $l(\beta_3)$ respectively. We glue $\beta_i$ with $b_i$, $i=1,2,3,$ by hyperbolic isometries. In this gluing, the resulting surface has genus two more than the genus of $\Sigma_0(G, d_t, \sigma_0)$ and number of boundary components three less than that of $\Sigma_0(G, d_t, \sigma_0)$.

\subsubsection{Glue II}\label{G2} 
Let $\beta$ be a boundary geodesic of $\Sigma_0(G, d_t, \sigma_0)$. We glue a hyperbolic 1-holed torus with boundary length $l(\beta)$ to the surface $\Sigma_0(G, d_t, \sigma_0)$ along $\beta$ by an isometry. The resultant surface will have genus one more than that of $\Sigma_0(G, d_t, \sigma_0)$ and the number of boundary components one less than that of $\Sigma_0(G, d_t, \sigma_0)$.

Now, assume that $\Sigma_0(G, d_t, \sigma_0)$ has $b$ boundary components. Using the division algorithm there are unique integers $q$ and $r$ such that $b=3q+r$, where $0\leq r \leq 2$. Then following the gluing procedure Glue I (see Subsection~\ref{G1}) for $q$ times and Glue II (see Subsection~\ref{G2}) for $r$ times, we obtain the desired closed hyperbolic surface denoted by $S(G, d_t, \sigma_0)$. 

\begin{rmk}
The genus of $S(G, d_t, \sigma_0)$ depends upon the fat graph structure $\sigma_0$. 
\end{rmk}


\section{Minimum Genus problem}\label{sec:4}
In this section, our goal is to prove Theorem\ref{thm:2}.

Let $(G, d)$ be a metric graph with degree of each vertex at least three and $\chi(G)$ denote the Euler characteristic of $G.$ We consider a fat graph structure $\sigma_0=\{\sigma_v|\; v\in V\}$ on $G$ and $\Sigma_0(G, d_t, \sigma_0)$, the hyperbolic surface with geodesic boundary obtained in Section~\ref{sec:3.5}. As $G$ is a spine of $\Sigma_0(G, d_t, \sigma_0)$, we have 
\begin{equation}\label{eq:7}
\chi(\Sigma_0(G, d_t, \sigma_0))=\chi(G),
\end{equation}
where $\chi(\Sigma_0(G, d_t, \sigma_0))$ denotes the Euler characteristic of $\Sigma_0(G, d_t, \sigma_0)$. The assumption $\deg(v)\geq 3$  and the relation $ 2|E| = \sum_{v\in V} \deg(v) \geq 3|V|$ implies that $\chi(G)< 0$.

\begin{lemma}\label{lem:5.1}
Let $\sigma_0$ and $\sigma_0'$ be two fat graph structures on $(G, d)$. Then the difference between the number of boundary components of $\Sigma_0(G, d_t, \sigma_0)$ and $\Sigma_0(G, d_t, \sigma_0')$ is an even integer, i.e., $\#\partial\Sigma_0(G, d_t, \sigma_0 ) - \# \partial  \Sigma_0(G, d_t, \sigma_0')$ is divisible by 2. 
\end{lemma}

The proof of Lemma~\ref{lem:5.1} is left to the reader.

The number of boundary components of a surface $F$ is denoted by $\#\partial F$. The genus of a fat graph $(G, \sigma_0)$ is the genus of the associated surface and denoted by $g(G, \sigma_0)$. Similarly, we define the number of boundary components of a fat graph and is denoted by $\#\partial (G, \sigma_0)$.

\begin{lemma}\label{prop:1}
Let $\sigma_0$ and $\sigma_0'$ be two fat graph structures on a metric graph $(G, d)$ such that $\#\partial(\Sigma_0(G, d_t, \sigma_0) )-\#\partial(\Sigma_0(G, d_t, \sigma_0'))=2$. Then we have $$g(S(G, d_t, \sigma_0'))\leq g(S(G, d_t, \sigma_0)).$$
\end{lemma}

\begin{proof}
Suppose that the genus and the number of boundary components of $\Sigma_0(G, d_t, \sigma_0)$ are $g$ and $b$ respectively. Then by Euler's formula and equation~\eqref{eq:7}, we have $2-2g-b=\chi(G)$ which implies that $b=2-2g-\chi(G)$. For the integer $b$, by the division algorithm, there exist unique integers $q$ and $r$ such that $b=3q+r,\;\;\textit{where}\;\; 0\leq r< 3.$ Therefore, by construction (see Section~\ref{sec:3.5}), the genus of $S(G, d_t, \sigma_0)$ is as following  $$g(S(G, d_t, \sigma_0) ) = g+2q+r.$$ Let us assume that the genus and number of boundary components of $\Sigma_0(G, d_t, \sigma_0')$ are $g'$ and $b'$ respectively. Then by Euler's formula and equation~\eqref{eq:7}, we have $b'=2-2g'-\chi(G)$. The hypothesis $b'=b-2$ of the lemma implies $g'=g+1.$ 

Now, we compute the genus of the closed surface $S(G, d_t, \sigma_0')$. There are three cases to consider as $b'=3q+r-2$ with $r\in \{0,1,2\}$. 

\subsection*{Case 1} $r=0$. 
In this case $b'=3(q-1)+1$. Thus the genus of $S(G, d_t, \sigma_0')$ is $g'+2(q-1)+1=g+2q$ which is equal to the genus of $S(\Sigma, d_t, \sigma_0)$. Therefore, the lemma holds with equality.

\subsection*{Case 2} $r=1$. 
In this case $b'=3(q-1)+2$. Therefore, genus of $S(G, d, \sigma_0')$ is $g+1+2(q-1)+2=g+2q+1$ which is  equal to the genus of $S(\Sigma, d_t, \sigma_0)$. Therefore, the lemma holds with equality. 

\subsection*{Case 3} 
The remaining possibility is $r=2$. In this case the genus of $S(G, d_t, \sigma_0)$ is $g+2q+2$. Now, $b'=b-2=3q$ implies that the genus of $S(G, d_t, \sigma_0')$ is $g'+2q=g+1+2q$. Therefore, we have $$g(S(G, d_t, \sigma_0'))=g(S(G, d_t, \sigma_0))-1<g(S(G, d_t, \sigma_0)).$$ 
\end{proof}

\begin{proof}[Proof of Theorem~\ref{thm:2}]
To find the essential genus of $(G, d)$, we consider a fat graph structure $\sigma_0$ on $G$ which gives maximum genus of $\Sigma_0(G, d_t, \sigma_0)$, equivalently minimum number of boundary components (see Lemma~\ref{prop:1}). For such a fat graph structure $\sigma_0$, the genus of $\Sigma_0(G, d_t, \sigma_0)$ is $\frac{1}{2}(\beta(G)-\zeta(G))$ which follows from Theorem 3 in~\cite{NHX}. Moreover, the number of boundary components of the fat graph $(G,\sigma_0)$ is $1+\zeta(G)$, which is equal to the number of boundary components of $\Sigma_0(G, d_t, \sigma_0)$. By the division algorithm, for the integer $1+\zeta(G)$, there are unique integers $q$ and $r$ such that $$1+\zeta(G)=3q+r,\;\;\textit{where}\;\; 0\leq r< 3.$$ Therefore, the genus of $S(G, d_t, \sigma_0)$ is $g_e(G) = \frac{1}{2}(\beta(G)-\zeta(G))+2q+r$, follows from the construction in Section~\ref{sec:3.5}. This proves the first part of the theorem.

Now, we focus on the proof of the remaining part of the theorem, i.e., we show that for any $g\geq g_{e}(G)$ the graph $(G, d_t)$ can be embedded on a closed hyperbolic surface of genus $g$. We define $g'=g-g_e(G).$ Let us consider the surface $S(G, d_t, \sigma_0)$ of genus $g_e(G)$ constructed above.  Now, there are two possibilities. 

\subsection*{Case 1} 
If the number of boundary components of $\Sigma_0(G, d_t, \sigma_0)$ is divisible by $3$, then we have a $Y$-piece, denoted by $Y(\beta_1', \beta_2', \beta_3')$, attached to $\Sigma_0(G, d_t, \sigma_0)$ along the boundary components $\beta_1, \beta_2, \beta_3$ by hyperbolic isometries in the construction of $S(G, d_t, \sigma_0)$ (see the construction in Section~\ref{sec:3.5}). We replace this $Y$-piece from $S(G, d_t, \sigma_0)$ by a hyperbolic surface $F_{g', 3}$ of genus $g'$ and three boundary components, again denoted by $\beta_1', \beta_2'$ and $\beta_3'$, of lengths $l(\beta_1), l(\beta_2)$ and $l(\beta_3)$ respectively. We denote the new surface by $S_{g'}(G, d_t, \sigma_0)$. 

\subsection*{Case 2} 
In this case, we consider the number of boundary components of $\Sigma_0(G, d_t, \sigma_0)$ is not divisible by $3$. Then there is a subsurface $F_{1, 1}$ of genus $1$ and a single boundary component $\beta'$, which we have attached to $\Sigma_0(G, d_t, \sigma_0)$ along the boundary component $\beta$ to obtain $S(G, d_t, \sigma_0)$. Now, we replace $F_{1, 1}$ by $F_{g'+1, 1}$, a hyperbolic surface of genus $g'+1$ and a single boundary component $\beta'$ of length $l(\beta)$, in $S(G, d_t, \sigma_0)$ and the obtained new surface is denoted by $S_{g'}(G, d_t, \sigma_0)$. 

The surface $S_{g'}(G, d_t, \sigma_0)$ has genus $g$ on which $(G, d_t)$ is isometrically embedded. 
\end{proof}

\section{Algorithm: Minimal embedding with minimum/maximum genus}\label{sec:6}
In this section, we study minimal essential embeddings and prove Proposition~\ref{thm:4}. We conclude this
section with Remarks~\ref{rmks:end} which provides an algorithm for minimal embedding with minimum
and maximum genus.

Let us consider a trivalent fat graph $(\Gamma,\sigma_0)$ with a vertex $v$ which is shared by three distinct boundary components. We construct a new fat graph structure to reduce the number of boundary components. 

\begin{lemma}
Let $(\Gamma,\sigma_0)$ be a 3-regular fat graph. If $\Gamma$ has a vertex which is common in three boundary components, then there is a fat graph structure $\sigma'_0$ such that $$\#\partial(\Gamma, \sigma'_0)= \#\partial(\Gamma, \sigma_0)-2.$$
\end{lemma}

\begin{proof}
Let $v$, a vertex, be in three distinct boundary components  of $\Gamma$. Assume that $v=\{ \vec{e}_i,\ i=1,2,3\}$ with $\sigma_v=(\vec{e}_1, \vec{e}_2, \vec{e}_3)$ (see Figure~\ref{fig:5}, left). Suppose $\partial_i, i=1,2,3$ are the boundary components given by (see Figure~\ref{fig:5}, right)
\begin{eqnarray*}
\partial_1 = \vec{e}_1P_1\cev{e}_3, \partial_2 = \vec{e}_2P_2\cev{e}_1 \text{ and } \partial_3 = \vec{e}_3P_3 \cev{e}_2,  
\end{eqnarray*}  
where $P_i$'s are finite (possibly empty) paths in the graph and $\cev{e}_i=\sigma_1(\vec{e}_i)$. We replace the order $\sigma_v=(\vec{e}_1, \vec{e}_2, \vec{e}_3)$ by $\sigma_v'=(\vec{e}_2, \vec{e}_1, \vec{e}_3)$ to obtain a new fat graph structure $\sigma_0'$. Then the boundary components of $(\Gamma, \sigma_0')$ given by $$\partial(\Gamma, \sigma_0')=(\partial(\Gamma,\sigma_0) \setminus\{\partial_i|i=1,2,3\})\cup \{\partial\},$$ where $\partial = \vec{e}_2P_2\cev{e}_1\vec{e}_3P_3\cev{e}_2\vec{e}_1 P_1\cev{e}_3=\partial_2*\partial_3*\partial_1.$ Here, $*$ is the usual concatenation operation. Therefore, the number of boundary components in $(\Gamma, \sigma_0')$ is the same as the number of boundary components in $(\Gamma, \sigma_0)$ minus two.

\tikzset{->-/.style={decoration={
  markings,
  mark=at position .5 with {\arrow{>}}},postaction={decorate}}} 
  
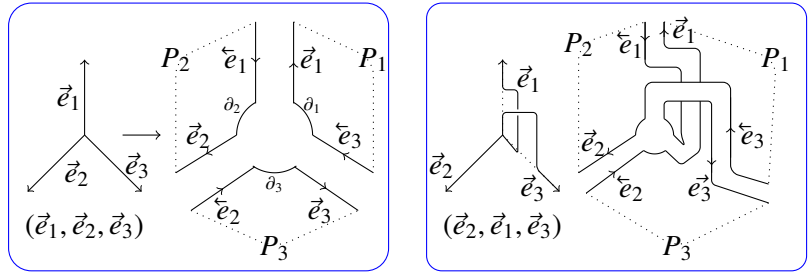
\begin{figure}[htbp]
\begin{center}
\begin{tikzpicture}
\draw [rounded corners=3mm,blue] (1.5, 1.75)--(-3.5, 1.75)--(-3.5, -1.8)--(1.5, -1.8)--cycle;
\draw [<->] (-2.5, 1) -- (-2.5, 0) -- (-3.25, -0.75); \draw [->] (-2.5, 0) -- (-1.75, -0.75); \draw [->] (-2, 0) -- (-1.5, 0); \draw (-2.5, -1.2) node {$(\vec{e}_1, \vec{e}_2, \vec{e}_3)$};
\draw (-2.7, 0.5) node {$\vec{e}_1$}; \draw (-2.6, -0.5) node {$\vec{e}_2$}; \draw (-1.8, -0.4) node {$\vec{e}_3$};

\draw [->-] (0.25, 0.42) -- (0.25, 1.5); \draw [->-] (-0.25, 1.5) -- (-0.25, 0.42); \draw [->-] (1.3, -0.5) -- (0.5, 0); \draw [->-](-1.1, -1) -- (-0.26, -0.42); \draw [->-](-0.5, 0) -- (-1.3, -0.5);  \draw [->-] (0.26, -0.42) -- (1.1, -1); \draw (0.5,0.4) node {\tiny{$\partial_1$}}; \draw (-0.55,0.4) node {\tiny{$\partial_2$}}; \draw (0,-0.65) node {\tiny{$\partial_3$}};

\draw [dotted, rounded corners=2mm] (0.25, 1.5) --(1.3,1) -- (1.3, -0.5); \draw (1.3, 1) node {$P_1$};

\draw [dotted, rounded corners=2mm]  (-1.3, -0.5)--(-1.3,1)-- (-0.25, 1.5); \draw (-1.3, 1) node {$P_2$};
\draw [dotted, rounded corners=2mm] (-1.1, -1) --(0, -1.5)node {$P_3$}-- (1.1, -1);

\draw [domain=0:60] plot ({0.5*cos(\x)}, {0.5*sin(\x)}); \draw [domain=120:180] plot ({0.5*cos(\x)}, {0.5*sin(\x)}); \draw [domain=235:305] plot ({0.5*cos(\x)}, {0.5*sin(\x)}); \draw (0.5, 1) node {$\vec{e}_1$}; \draw (-0.5, 1) node {$\cev{e}_1$}; \draw (-1,0) node {$\vec{e}_2$}; \draw (1,0) node {$\cev{e}_3$}; \draw (-0.6,-1) node {$\cev{e}_2$}; \draw (0.6,-1) node {$\vec{e}_3$}; 
\draw [rounded corners=2mm,blue] (2, 1.75)--(2, -1.8)--(7, -1.8)--(7,1.75)--cycle;

\draw [domain=0:60] plot ({5+0.25*cos(\x)}, {0.25*sin(\x)}); \draw [domain=120:180] plot ({5+0.25*cos(\x)}, {0.25*sin(\x)}); \draw [domain=235:305] plot ({5+0.25*cos(\x)}, {0.25*sin(\x)});

\draw [->-] (4.75, 0) -- (4, -0.5); \draw [->-] (4.1, -0.77) -- (4.856, -0.205); \draw [rounded corners = 1mm] (4.875, 0.216) -- (4.875, 0.7)-- (6, 0.7) -- (6, -0.5) -- (6.5, -0.65); \draw [rounded corners=0.5mm] (5.125,0.216) -- (5.125, 0.45) -- (5.75,0.45) -- (5.75, -0.65) -- (6.5, -0.9); \draw [rounded corners=0.75mm] (4.875,1.5) -- (4.875, 0.9) -- (5.35, 0.9) -- (5.35, 0.7) ; \draw [rounded corners=0.75mm](5.125, 1.5) -- (5.125,1.15) -- (5.6, 1.15) -- (5.6, 0.7); \draw [rounded corners=0.5mm] (5.143, -0.205) -- (5.35, -0.4) -- (5.6, -0.21) -- (5.6, 0.45); \draw [rounded corners=0.5mm] (5.25, 0) -- (5.4, -0.15) -- (5.35, 0) -- (5.35, 0.45);
\draw [->-] (6, -0.3) -- (6, 0.5);\draw (6, 0) node [right] {$\cev{e}_3$}; \draw (4.2, -0.1) node {$\vec{e}_2$}; \draw (4.7, -0.7) node {$\cev{e}_2$}; \draw [->-] (4.875, 1.5) -- (4.875, 1); \draw (4.7, 1.2)node {$\cev{e}_1$}; \draw [->-] (5.125, 1.3) -- (5.125, 1.5); \draw (5.4, 1.4)node {$\vec{e}_1$}; \draw [->-] (5.75, -0.3) -- (5.75, -0.5); \draw (5.6, -0.7) node {$\vec{e_3}$};

\draw [dotted, rounded corners=2mm]  (5.125, 1.5)--(6.6,1)node {$P_1$}-- (6.5, -0.65);
\draw [dotted,rounded corners=2mm] (6.5, -0.9) --(5.3, -1.5)node {$P_3$}-- (4.1, -0.77);

\draw [dotted, rounded corners=2mm] (4, -0.5) -- (4, 1.2)node {$P_2$} -- (4.875, 1.5);

\draw [->](3, 0) -- (2.25,-0.75); \draw [dotted] (3, 0) -- (3.75, -0.75); \draw [dotted](3, 1) -- (3, 0); \draw [rounded corners=0.5mm, <-] (3, 1) -- (3, 0.6) -- (3.2, 0.6) -- (3.2, 0.34); \draw [rounded corners=0.3mm] (3.2, 0.26) -- (3.2, -0.24) -- (3, 0); \draw [rounded corners=0.5mm, <-] (3.75, -0.75) -- (3.45, -0.45) -- (3.45, 0.3) -- (3, 0.3) -- (3, 0);

\draw (3.35, 0.75) node  {$\vec{e}_1$}; \draw (2.2, -0.4) node {$\vec{e}_2$}; \draw (3.4, -0.7) node {$\vec{e}_3$}; \draw (3, -1.2) node {$(\vec{e}_2, \vec{e}_1, \vec{e}_3)$};

\end{tikzpicture}
\caption{Change of cyclic order at a vertex}\label{fig:5}
\end{center}
\end{figure}
\end{proof}

\begin{proof}[Proof of Proposition~\ref{thm:4}]
Let us consider $v_0=\{\vec{e}_1, \vec{e}_2, \dots, \vec{e}_k\}$ be a vertex which is shared by at least three boundary components. We assume that the cyclic order at $v_0$ is given by $$\sigma_{v_0}=\left( \vec{e}_1, \vec{e}_2, \vec{e}_3, \dots, \vec{e}_i, \vec{e}_{i+1}, \dots, \vec{e}_k\right),$$ where $k\geq 3$, $3 \leq i\leq k$ and $\vec{e}_{k+1}=\vec{e}_1$. We can choose three boundary components $b_i, i=1, 2, 3,$ such that there is an edge $e_2=\{\vec{e}_2, \cev{e}_2\}$ with $\vec{e}_2$ in $b_1$ and $\cev{e}_2$ in $b_2$. We can write $b_1=\vec{e}_2 P_1 \cev{e}_1, b_2 = \vec{e}_3 P_2 \cev{e}_2$ and $b_3 = \vec{e}_{i+1} P_3 \cev{e}_i$, where $P_j$'s are some paths (possibly empty) in the fat graph. Now, we consider a new cyclic order at $v$, given by $\sigma'_v=(\vec{e}_1, \vec{e}_3, \dots, \vec{e}_i, \vec{e}_2, \vec{e}_{i+1},\dots, \vec{e}_k)$. Then, in the new fat graph structure $\sigma_0'$, the boundary components of $(G, \sigma_0')$ are $\left(\partial(G, \sigma_0)\setminus \{b_1, b_2, b_3\}\right) \cup\{b\}$, where $b=b_1*b_2*b_3$. Therefore, we have $$\#\partial(G, \sigma_0')=\#\partial(G, \sigma_0)-2.$$
\end{proof}

\begin{rmks}\label{rmks:end}
\begin{enumerate}
\item One can obtain minimal embedding by applying Proposition~\ref{thm:4}. let us consider a fat graph $(G, \sigma_0)$. If $v$ is a vertex shared by $k\;\leq 2$ boundary components, then it follows from Lemma~\ref{lem:5.1} and Lemma~\ref{prop:1} that there is no replacement of the cyclic order $\sigma_v$ (keeping the cyclic order on other vertices unchanged) to reduce the number of boundary components. If there is a vertex $v$ which is shared by at least three boundary components, then one can replace the fat graph structure by applying Proposition~\ref{thm:4} to reduce the number of boundary components by two. Therefore, by repeated application of Proposition~\ref{thm:4}, we can obtain a fat graph structure on the graph $G$ that provides the essential genus $g_e(G)$.

\item Using Proposition~\ref{thm:4}, in the reverse way, we can obtain a fat graph structure which provides the maximal genus $g_e^{\max}(G)$ of a minimal embedding. Namely, if there is a vertex $v$ with cyclic order $\sigma_v=(\vec{e}_2, \vec{e}_3\dots,\vec{e}_i, \vec{e}_1, \vec{e}_{i+1}, \dots \vec{e}_k)$ and a boundary component $\partial$ of the form $\partial = \vec{e}_2 P_1 \cev{e_1} \vec{e_3} P_2  \cev{e}_2 e_{i+1} P_3 \cev{e}_i$, where $P_j$'s are some paths in $G$, only then, we can replace $\sigma_0$ by $\sigma_v' = (\vec{e}_1, \vec{e}_2, \vec{e}_3, \dots, \vec{e}_i, \vec{e}_{i+1}, \dots, \vec{e}_k)$ to obtain a new fat graph structure $\sigma_0'$, such that $$\#\partial(G, \sigma_0')=\#\partial(G, \sigma_0)+2.$$   By repeated use of Proposition~\ref{thm:4}, we can obtain a fat graph structure on $G$ which provides $g_e^{\max}(G)$. 
\end{enumerate}

\end{rmks}

\end{document}